\documentclass{amsart}
\usepackage[utf8]{inputenc}

\usepackage{amsmath,amsthm,verbatim,amssymb,amsfonts,amscd,graphicx,calc}
\usepackage{xcolor}
\usepackage{graphics}

\usepackage{euscript, enumerate,hhline,array}
\usepackage{url,hyperref}

\newcommand{\p}{\partial}

\newcommand{\la}{\langle}
\newcommand{\ra}{\rangle}

\newcommand{\e}{\varepsilon}

\newcommand{\eps}{\varepsilon}

\newcommand{\be}{\begin{equation}}
\newcommand{\ba}{\begin{aligned}}
\newcommand{\bee}{\begin{equation*}}
\newcommand{\ee}{\end{equation}}
\newcommand{\ea}{\end{aligned}}
\newcommand{\eee}{\end{equation*}}
\newcommand{\bea}{\begin{equation} \begin{aligned} }
\newcommand{\eea}{\end{aligned}\end{equation} }

\newcommand{\dist}{{\rm dist}\, }

\theoremstyle{plain}
\newtheorem{theorem}{Theorem}[section]

\newtheorem{corollary}[theorem]{Corollary}

\newtheorem{lemma}[theorem]{Lemma}
\newtheorem{proposition}[theorem]{Proposition}

\theoremstyle{remark}

\newtheorem{remark}[theorem]{Remark}

\theoremstyle{definition}
\newtheorem{definition}[theorem]{Definition}

\numberwithin{equation}{section}

\title{Monge--Amp\`ere equations with right-hand sides of polynomial growth}

\author{Beomjun Choi}

\address{BC: Department of Mathematics, POSTECH, 77 Cheongam-ro, Nam-gu, Pohang, Gyeongbuk 37673, Republic of Korea}
\email{bchoi@postech.ac.kr}

\author{Kyeongsu Choi}
\address{KC: School of Mathematics, Korea Institute for Advanced Study, 85 Hoegiro, Dongdaemun-gu, Seoul 02455, Republic of Korea}
\email{choiks@kias.re.kr}

\author{Soojung Kim}
\address{SK: Department of Mathematics, Soongsil University,  Seoul 06978,    Republic of Korea}
\email{soojungkim@ssu.ac.kr}

\date{}

\begin{document}

\begin{abstract} We study the regularity and the growth rates of solutions to two-dimensional Monge--Amp\`ere equations with the right-hand side exhibiting polynomial growth. Utilizing this analysis, we demonstrate that the translators for the flow by sub-affine-critical powers of the Gauss curvature are smooth, strictly convex entire graphs. These graphs exhibit specific growth rates that depend solely on the power of the flow.
\end{abstract}
 
\maketitle

\section{Introduction}

A one-parameter family of smooth embeddings $X:M^2 \times [0,T)\to \mathbb{R}^3$ with complete convex images $\Sigma_t :=X(M^2,t)$ is a solution to the $\alpha$-Gauss curvature flow  ($\alpha$-GCF) if  $X_t=-K^\alpha \nu$ holds, where $K$ and $\nu$ denote the Gauss curvature and  the unit normal vector pointing the outside of the convex hull of $\Sigma_t$, respectively.  A complete convex hypersurface $ \Sigma\subset \mathbb{R}^3$ is said to be  a translating soliton for  the  $\alpha$-GCF travelling with the velocity $  \omega\in\mathbb{R}^3$ if
 \begin{equation}\label{eq-translator}   
  K^\alpha= \la  \omega , -\nu \ra.
 \end{equation}    
In this paper,  we will assume  $\omega=\bf{e_3}$, namely translators are moving in  the $+x_3$ direction with  the  unit speed.  Then, each translator $\Sigma $   can be represented   as  
 \be\label{eq-repre-translator-graph}
\Sigma = \p \{ (x_1,x_2,x_3)\in \Omega\times  \mathbb{R}:  x_3 >u(x_1,x_2) \}
\ee 
for a convex function $u$ defined on a convex   domain  $\Omega\subset \mathbb{R}^2$.  Note that    $\Sigma $  is    the union of the graph of $u$ over $\Omega $ and  some part of  $\partial \Omega\times   \mathbb{R}  $.   If  $\Sigma$  is smooth, then 
$u$ is a classical solution to the following Monge--Amp\`ere equation
 \bea \label{eq-translatorgraph}
 \det D^2 u = (1+|Du|^2)^{2-\frac1{2\alpha}}, \eea 
with the boundary condition that $|Du|\to +\infty$ as $x\to \partial \Omega$.

It is well-known that for smooth convex closed initial hypersurface, the $\alpha$-GCF with $\alpha>\frac{1}{4}$ converges to a round sphere after rescaling and the $\frac{1}{4}$-GCF converges to an ellipsoid after rescaling. See \cite{andrews1999gauss,andrews2011surfaces,andrews2016flow,brendle2017asymptotic,guan2017entropy}. However, a closed $\alpha$-GCF with $\alpha \in (0,\frac{1}{4})$ generically develops a type II singularity. Moreover, translators are typical models of type II singularities. Suppose that a $\alpha$-GCF $\Sigma_t$ develops a type II singularity $(x_0,t_0)$, and a sequence of rescaled flows $\Sigma^i_t:=\lambda_i(\Sigma_{t_0+\lambda_i^2t}-x_0)$, with $\lambda_i\to +\infty$, converges to a translating flow $\overline{\Sigma}_t:=\Sigma+t\omega$. If there is no compactness result with uniform estimates for $\Sigma^i_t$, then it would converge in a weak sense. Hence, it is worth to study the regularity of the limit translator $\overline{\Sigma}$ defined in a weak solution.

 In this paper, we not only show that a weak solution to \eqref{eq-repre-translator-graph} with $\alpha \in (0,\frac{1}{4})$ is smooth, but also reveals that it is an entire function increasing with the polynomial growth rate $|x|^{\frac{1}{1-2\alpha}}$.

 \begin{theorem}[Growth rate]\label{thm-smooth}
Let $\Sigma $ be a translator under the $\alpha$-GCF in the geometric Alexandrov sense\footnote{See Definition \ref{def-translatoralexandrov}} for $\alpha \in (0,1/4)$. 
 Then, $\Sigma$ is the   graph of an entire smooth strictly convex function $u: \mathbb{R}^2\to\mathbb{R}$. Moreover, there is  a constant $C>1$ depending only on $\alpha$ such that 
\begin{equation}\label{eq-440}
C^{-1}\le  \liminf_{|x|\to \infty}  |x|^{-\frac{1}{1-2\alpha}}u(x) \le\limsup_{|x|\to \infty}  |x|^{-\frac{1}{1-2\alpha}}u(x)  \le C     
\end{equation} 
holds in $\mathbb{R}^2$.
\end{theorem}

\bigskip

To prove this theorem, we work with the dual equation via the Legendre transform. Utilizing the results in \cite{DSavin}, we show that the Legendre dual function  exhibits the homogeneous growth rate $u^*(y)\sim |y|^{\frac{1}{2\alpha}}$.      
For     a convex function $u$ over a convex domain  $\Omega\subset \mathbb{R}^2$,    the Legendre transform  $ u^*: \mathbb{R}^2\to \mathbb{R}\cup\{
+\infty\}$  of $u$  is defined  by 
\bea   u^*(y) := \sup_{x\in \Omega} \,\left\{  \langle y,x\rangle  - u(x)\right\}. \eea 
 If    a  strictly convex  function $u$ solves \eqref{eq-translatorgraph} on $ \Omega$ in the  classical sense,   then the Legendre  transform   $u^*$   of $u$ becomes a classical solution to the dual equation   
  \bea \label{eq-dualtranslator-1} \det D^2 u^* = (1 +|x|^2)^{\frac{1}{2\alpha}-2} \eea
on $ \mathrm{int} (\Omega^*) $,  where   
  \bea \label{eq-Omegadual}\Omega^* = \{y\in \mathbb{R}^2 \,:\, u^*(y)<\infty \}.\eea   
We state the growth rate theorem for solutions to a certain class of Monge--Amp\`ere equations, which is also available for \eqref{eq-dualtranslator-1}.

\bigskip

For the purpose of brevity, from now on we say that a constant is \textit{universal} if it depends only on $\alpha \in (0,\frac{1}{4})$ and the \textit{doubling constant} of the measure $f(x)\,dx$. See Definition \ref{def-doubling} for the doubling constant.

\begin{theorem}\label{thm-equalgrowth}
For a doubling measure $f(x)\, dx$, let $v$ be an Alexandrov solution to 
\begin{equation}\label{eq-standard.MA}
    \det D^2v=f(x),
\end{equation}
on an open set $\mathcal{O}  \subset \mathbb{R}^2$, such that the sub-level set  $\{x\in\mathcal{O}  :v(x) \le t\}$ is  compactly included in $\mathcal{O}$ for every $t\in \mathbb{R}$. For $\alpha\in (0,1/4)$, there are some universal constants $C> 1$ and $\varepsilon \in (0,1)$ with the following significance: if

\begin{equation}\label{eq-RHS.conv2}
\limsup _{x\to \infty} ||x|^{4-\frac{1}{\alpha}}f(x)-1 |\le \varepsilon,
\end{equation}
then $v$ is an entire solution, namely $\mathcal{O} =\mathbb{R}^2$, satisfying
\begin{equation}\label{eq-CR} 
C^{-1}\le  \liminf_{|x|\to \infty} \, |x|^{- \frac 1{2\alpha}}v(x)\, \le \,\limsup_{|x|\to \infty} \, |x|^{- \frac 1{2\alpha}}v(x)  \le C.     
\end{equation} 
\end{theorem}

\begin{remark}
For the Legendre dual translator equation \eqref{eq-dualtranslator-1}, the condition \eqref{eq-RHS.conv2} is satisfied arbitrary $\eps>0$. Also, for $\alpha \in (0,1/4)$, $(\eta+|x|^2)^{\frac{1}{2\alpha}-2}dx$ for $\eta \in [0,1]$ is a doubling measure and its doubling measure depends only on $\alpha$. Thus, the constant $C$ in \eqref{eq-CR} depends only on $\alpha$.
\end{remark}


\bigskip

Finally, we also provide a quantitative variant of the above theorems, which will play a crucial role in the following paper \cite{choi2021translating} for the classification of translators. The following corollary roughly says if the eccentricity is controlled at a sufficiently high level, then it is also controlled at higher levels in a universal way.

To state the result, we recall the notion of sections. Let  $v:\mathcal{O}\to \mathbb{R}$ be a convex function. For    $x_0\in \mathcal{O} $, $ p\in \p v(x_0)$  and $t>0$,   a section  of $v $ at height $t$ is  denoted by 
 \begin{equation}
  S_{t,x_0,p}^v := \{ x\in\mathcal{O}\,:\,  v(x) \le v(x_0) + p \cdot (x-x_0) +t \}.   
 \end{equation}
 Note that if  $v$ is differentiable at $x_0$, then  $\p v(x_0)$ has only one element.   By abusing notations,     we shall    write  $  S_{t,x_0} $  for $ S_{t,x_0,p}^v $ and $S_{t}=S_{t,0}$ whenever $v$ and $p$ are obvious and the base point is the origin, respectively.

We also denote by 
\begin{equation}
S^v_{t,x_0,p}\sim A,
\end{equation}
if the eccentricity of a section $S^v_{t,x_0,p}$ is proportional to $|A|$ of a matrix $A$ with $\det A=1$. See Definition \ref{def-eccen-section} for the eccentricity and the definition of $\sim$ for a given doubling measure $f(x)\, dx$..

\begin{corollary} \label{cor-quantitativestability}
Let $f(x)\, dx$ be a doubling measure. There are  universal constants $\varepsilon,C_0,C_1$ and a constant $l_0$ depends only on $f,\alpha$ with the following significance. Let $v$ be an entire Alexandrov solution to \eqref{eq-standard.MA}  with $v(0)=0$, $0 \in \partial v(0)$  satisfying \eqref{eq-RHS.conv2} for some $\alpha \in (0,1/4)$. If $S_{l'} \sim A_{l'} $ and $|A_{l'}|\le  M$ for some ${l'}\ge l_0$ and $M\ge C_0$, then $|A_l|\le C_1M$ for all $A_l$ with $S_l \sim A_l$  and $l\ge l'$.
\end{corollary}

This follows from an application of results in Savin and Daskalopoulos \cite{DSavin}. In large scales, the dual equation looks like $\det D^2 v= |x|^{\frac{1}{\alpha}-4} $ with $\alpha <\frac{1}{4}$. If the dual $v$ does not have a homogeneous rates, the level sets (curves) will degenerate. After suitable affine transforms, the $v$ will behave like a solution to $\det D^2 v= |x_1|^{\frac{1}{\alpha}-4}$. \cite[Section 3]{DSavin} shows the level curves of every such solution becomes more eccentric as the level gets smaller. Here is a rough idea of our proof: 
when $v$ is globally defined entire solution, this implies that the level curves can not become more eccentric as the level gets higher as otherwise the high eccentricity at high level implies even higher eccentricity at lower levels.

This work is a part of our series of research on the existence, regularity, and uniqueness of translating solitons to the sub-affine-critical powers of the Gauss curvature flow. These questions have been intensively studied for decades, and here we give some references. For the classical equation $\det D^2u=1$ in all dimensions, the convex entire solutions exhibit quadratic growth rates, and moreover, such solutions are necessarily quadratic polynomials \cite{caffarelli1995topics, cheng1986complete, calabi1958improper, jorgens1954losungen, nitsche1956elementary}. See also extensions of this result by Caffarelli--Y. Li \cite{caffarelli2003extension} to the case when the right-hand side is constant outside a compact set, and by Y. Li--S. Lu \cite{caffarelli2004liouville,li2019monge} to the case with a periodic right-hand side.

In the case of translators for the $\alpha$-GCF with $\alpha>1/2$, the questions were thoroughly answered in all dimensions by J. Urbas \cite{Urbas1988, U98GCFsoliton}. The translators in this range of $\alpha$ are graphs over bounded domains, and for each fixed bounded domain, \cite{choi2020uniqueness} showed the translator is the unique complete ancient solution to the flow. The translators also exhibit forward-in-time stability \cite{choi2018convergence, choi2019convergence}.

For $\alpha<1/2$ in all dimensions, H. Jian--X.J. Wang \cite{JW} constructed infinitely many solutions for $\alpha\leq \frac{1}{2}$, and also showed that $\Omega$ should be the entire space $\mathbb{R}^n$ for $\alpha <\frac{1}{n+1}$. Moreover, they showed certain estimates on the growth rates $ |x|^{1+a} \lesssim u^*(x) \lesssim |x|^{1+b}$ for large $|x|$ with some positive constants $a$ and $b$. To our knowledge, however, neither a classification result nor a sharp estimate on the growth rate is known for these cases in any dimension $n\ge2$. It should be noted that for the degenerate equation $\det D^2 u = |x_1|^{\beta }$ on $\mathbb{R}^2$ for $\beta>-1$ (which corresponds to a blow-down equation of our equation when the eccentricity becomes large), the classification of entire solutions was settled by Jin-Xiong in \cite{jin2014liouville}.

\section{Preliminary}\label{sec-prelim}
 
 In this section, we recall and compare weak solutions to Monge--Amp\`ere equations and $\alpha$-Gauss curvature flow translator equations \eqref{eq-translator}. Also, we show the compactness of level sets of the Legendre dual of weak translators.

 \smallskip

To begin with we recall the definition of  Alexandrov solutions  to the Monge--Amp\`ere equation. See \cite{MA-figalli-book}.  
    
 \begin{definition} [Alexandrov solution] \label{def-alexandrov}
A convex function $u$ on a convex  open  set $\Omega\subset \mathbb{R}^2$ is said to be  a solution to \eqref{eq-translatorgraph} on $\Omega$ in the   Alexandrov sense if for any Borel set $U \subset    \Omega  $, there holds 
\bea \label{eq-C13} \mu_u (U) =  \int_ U (1+|Du|^2)^{2-\frac{1}{2\alpha}}dx ,\eea
where 
\bea \mu_u (U)=  | \p u(U)|\quad\hbox{for every Borel set $U \subset   \Omega$},\eea 
and
\bea \p u (U) = \cup_{x\in U}\p u(x), \eea   
and $\p u(x)$ is  the sub-differential of    $u: \Omega \to \mathbb{R}$ at $x\in \Omega$. Notice that $|Du|$ is well-defined almos everywhere on $\Omega$, as a convex function  $u$ is locally Lipschitz on $\Omega$. Consequently, the right-hand side of \eqref{eq-C13} makes no confusion.
\end{definition}

 Next, let us recall  generalized notions  of   translating solitons for the  $\alpha$-Gauss curvature flow introduced by    Urbas in \cite{U98GCFsoliton}.  

\begin{definition} [Translator] \label{def-translatoralexandrov}
We say that   a   complete     convex hypersurface $\Sigma $ is a translator for the  $\alpha$-Gauss curvature flow in the geometric Alexandrov sense if for any Borel set $E \subset \Sigma$, there holds 
\bea  \label{eq-C10}\mathcal{H}^2 (G(E)) = \int_E (-\nu_{3})^{\frac1 \alpha}.\eea
Here, $\nu= (\nu_1,\nu_2,\nu_3)$ is the outward pointing unit normal vector, which is $\mathcal{H}^2$-a.e. well-defined on $E$.  For a Borel set $E\subset \Sigma$, $G(E)$ is the generalized Gauss image given by \bea G(E) = \cup_{p\in E} G(p) ,\eea where $G(p)$ is the set of outward pointing unit normals of supporting hyperplanes of $\Sigma$ at $p$. 
\end{definition}

\smallskip
 
The following proposition shows the relation between the previous two notions of weak solution.

\begin{proposition} \label{prop-C4}
Let $\Sigma $ be a translator  under  the $\alpha$-GCF  in the     geometric Alexandrov sense represented  by \eqref{eq-repre-translator-graph}.  Then the following statements hold.  
\begin{enumerate}[(a)]
\item $ \mathcal{H}^2\left( G(\Sigma \cap (\p \Omega \times \mathbb{R}))\right)=0 $. i.e. the outward normals supporting on  $\p \Omega \times \mathbb{R}$ are  negligible. \item    $u$  restricted to $  \mathrm{int}( \Omega)$ solves \eqref{eq-translatorgraph} on  $  \mathrm{int}( \Omega)$  in the    Alexandrov sense. 
\end{enumerate}
\end{proposition}
\begin{proof} 

(a) The first part is immediate from   Definition \ref{def-translatoralexandrov}. Indeed,  for all point $p\in \Sigma \cap (\p \Omega \times \mathbb{R})$, $G(p)$ contains a normal vector with $\nu _3=0$,  i.e., $\nu_3=0$  $\mathcal{H}^2$-a.e. on $G(\Sigma \cap (\p \Omega \times \mathbb{R}))$.
\smallskip

(b) Next, let $U\subset \mathrm{int}(\Omega) $ be a given Borel set, and let  $E \subset \Sigma$  be a Borel set   such that $U$ is  the projection of $E  $ onto the plane.  
 Note that $\varphi(\p u(U) ) = G(E)$,  where $\varphi:\mathbb{R}^2 \to \mathbb{S}^2 $ 
  is given as  $  \varphi(y)=  \frac{(y,-1)}{(1+|y|^2)^{1/2}}$. By the area formula, we have 
\be
\mu_u (U)= \int_{ {G(E)} } \sum_{q\in \varphi^{-1}\{p\}}[J\varphi(q)]^{-1}  d\mathcal{H}^2(p)= \int_{G(E)} (-\nu_3(p)) ^{-3} d\mathcal{H}^2(p)  
\ee
since $J\varphi(y)=  (1+|y|^2)^{-3/2}$.  Here, we used that $\p u(x)$ is uniquely defined a.e. on $ \Omega$ as $Du(x)$.
Meanwhile, it follows from    \eqref{eq-C10} that 
\bea \int_{G(E)} (-\nu_3(p)) ^{-3} d\mathcal{H}^2(p) = \int_E (-\nu_3)^{-3 +\frac 1\alpha} = \int _U     (1+|Du|^2)^{2-\frac1{2\alpha}} dx  \eea
since  the Jacobian of a map $x\mapsto (x,  Du(x))$  is  $(1+|Du|^2)^{1/2}$    a.e. on  $  U$.  
Combining  the estimates above,  we deduce that $ \mu_u (U)= \int_U (1+|Du|^2)^{2-\frac1{2\alpha}} dx$, which finishes the proof.
\end{proof}

\bigskip

Given a translator $\Sigma $ with the graph representation $u$ in \eqref{eq-repre-translator-graph}, we can consider the Legendre dual $u^*$ and the dual domain $\Omega^*$ as in \eqref{eq-dualtranslator-1} and \eqref{eq-Omegadual}.

\begin{proposition}\label{prop-C51}
 Let $\Sigma $ be a translator  under  the $\alpha$-GCF  in the     geometric Alexandrov sense represented  by \eqref{eq-repre-translator-graph}. Then, 
  the   dual function $u^*$ is an Alexandrov solution to the Legendre dual equation $ \det D^2 u^* = (1+|y|^2)^{\frac{1}{2\alpha}-2}$ on $\mathrm{int}( \Omega^*)$.  Namely, it  holds  that
\be\label{eq--C51}
\mu_{u^*}(V) = \int_V (1+|y|^2)^{\frac{1}{2\alpha}-2} dy 
\ee  for any  Borel set $V\subset\mathrm{int}(\Omega^*)$. 
\end{proposition}

 \begin{proof}   We first recall a preliminary result from the analysis of convex functions.  Let us denote the domain of $u^{**}= (u^*)^*$  by
\begin{equation}
\Omega^{**}:= \{x\in \mathbb{R}^2 \, : \,  \sup_{y\in \Omega^*} \left\{\la x,y  \ra-u^*(y)\right\}<\infty \}.
\end{equation} 
Note that $u^{**}$ is the convex envelope (convexification) of $u$. Since $u$ is convex, it holds that $u=u^{**}$ on $\Omega$. Due to an ambiguity on the choice of $\Omega$,  $\Omega^{**}$ can possibly be slightly larger as $\Omega \subset \Omega^{**}\subset \mathrm{cl}(\Omega)$. We may replace $u$ and $\Omega$ by the lower-semi continuous extension on the maximal domain  in order to make $\Omega=\Omega^{**}$, which  does not change $u^*$ and $\Omega^*$; see \cite[Proposition 5.8]{MR2459454}.

 To prove \eqref{eq--C51},  it suffices to show that  for any  Borel set   $V \subset  \mathrm{int}( \Omega^*)$,  
  \be 
  |V|=  \int_V (1+|y|^2)^{2-\frac 1{2\alpha}} d\mu_{u^*}(y) .
  \ee 
We observe that 
\begin{equation}\label{eq53}
  \int_V (1+|y|^2)^{2-\frac 1{2\alpha}} d\mu_{u^*}(y) = \int_0^\infty  \mu_{u^*}(V_s)ds,
\end{equation} 
 where
\begin{equation}
V_s:=V\cap \{y:\in \text{int}(\Omega^*):  (1+|y|^2)^{2-\frac 1{2\alpha}} \ge s \}.
\end{equation}
Then the integrand can be rewritten as 
\begin{align}
\mu_{u^*}(V_s)= \left| \{x\in \mathrm{int}(\Omega): u \text{ is differentiable at } x, \text{ and } Du(x) \in V_s \}\right|.
\end{align}
Notice that   for $x\in \Omega$ and $y\in \Omega^*$, $ x \in \p u^*(y)$  is equivalent to $ y\in \p u(x)$.  
  Letting 
  \be
  U':=\{ x\in\mathrm{int}(\Omega):\,\text{ $u$ is differentiable at } x,\, D u (x)\in    V   \},
  \ee
and $V':=Du (U')$,
it holds from  \eqref{eq53} and  (b) of Proposition \ref{prop-C4} that 
  \bea  \label{eq53-1}
   \int_V (1+|y|^2)^{2-\frac 1{2\alpha}} d\mu_{u^*}(y)
    &= \int_{ U'} (1+|Du|^2)^{2-\frac{1}{2\alpha}}dx = |V'| .
    \eea      
 Note that  $V'\subset V$ and $ V\setminus V' \subset \p u (\Omega\setminus  \mathrm{int}(\Omega))$ $\cup$ $ \p u( W)$, where $W:=\{ x\in \Omega: \text{$u$ is non-differentiable at $x$}\}$.  So, it remains to prove that $ |\p u (\Omega\setminus  \mathrm{int}(\Omega)) |=0=|\p u ( W) |$. 
 	 We first observe that $ |\p u (\Omega\setminus  \mathrm{int}(\Omega)) |=0$ in light of   (a) of Proposition \ref{prop-C4}. Next,    (b) of Proposition \ref{prop-C4} and \eqref{eq-C13} require that $d\mu_u $ is absolutely continuous with respect to $dx$. Then  we obtain that  $|\p u (W) | =0$ since    $u$ is differentiable a.e. on $\Omega$. This shows that $|V'|=|V|$, which   combined with \eqref{eq53-1} finishes the proof.  
 \end{proof}

\bigskip

In the following theorem, we prove that level sets of the dual $u^*$ are compact. 

\begin{theorem} \label{thm-section}
Let $\Sigma $ be a translator  under  the $\alpha$-GCF  in the     geometric Alexandrov sense represented  by \eqref{eq-repre-translator-graph}. Then the following statements hold.  
 \begin{enumerate}[(a)] 
\item 
If  $\{y_j\}$ is a sequence in $  \mathrm{int}( \Omega^*)$ converging to $y \in \p \Omega^*$,  then  
\be
\lim_{j\to \infty}u^*(y_j)=\infty.
\ee
\item For  any   $y' \in \mathrm{int}(  \Omega^*)$ and  $x'\in \p u^*(y')$ with $x'\in \mathrm{int}(  \Omega)$, the section 
\bea 
S_{t,y',x'}^{u^*}  =\left\{ y\in \Omega^* \,:\, u^* (y) \le u^*(y')+\langle x', y-y'\rangle +  t \right\}
\eea 
is compact and contained in $\mathrm{int}( \Omega^*)$ for all $t>0$.  
\end{enumerate}
 \end{theorem}

\begin{proof}
(a) 
 Suppose to  the contrary that  there is a sequence  $y_j$ in $\mathrm{int}( \Omega^*)$ converging to $ y\in \p \Omega^*$ such that $u^*(y_j)=:L_j \to L $ for some $L\in \mathbb{R}$. 
 (Note that $L$ can not be $-\infty$ due to the    convexity.)  
 Then we obtain that  
\be\label{eq-lem3-barr1}
   \langle y,x\rangle -L=\lim_{j\to \infty }\left\{ \langle y_j,x\rangle -L_j \right\}\le u(x)\quad\hbox{ for  any $x\in \Omega$}. 
   \ee

 On the other hand,  we observe that for any arbitrary point $p$ inside the convex hull of the soliton, the half ray $\{p + t(y,|y|^2)\}_{t\ge 0}$ starting at $p$ should be included inside   the convex hull of the soliton since  $ y\in \p \Omega^*$.  
 After a translation, we may assume that $\Omega$ contains a ball of radius $r>0$ centered at $0$, say $B_r$. 
 Let  $M:= \sup_{x\in B_r} u(x) .$
  By considering the rays starting from a  floating disk $B_r \times \{M\}$ as barriers,   it holds that for each $t>0$,
 \bea\label{eq-lem3-barr2}
  u  \le M+ t|y|^2 \quad \text{    on  each translated disk }    B_r +t y.
  \eea 
 By \eqref{eq-lem3-barr1} and \eqref{eq-lem3-barr2}, there holds a two-sided bound:   for each $t>0$, 
\be
 (t|y|-r )|y| -L \le u (x) \le t|y|^2 +M\quad\hbox{  for  any $x\in   B_r +t y$}. 
\ee 
  Thus there is a constant $C=C(L,M,|y|,r )$ such  that 
\be   |Du | \le C \quad\hbox{  on a half-strip region $ \cup_{t>0}\big(B_{r/2}+ty \big)  $}
.
\ee
Letting    $E \subset \Sigma$  to be  the graph of $u$ over $A:=\cup_{t>0}\big(B_{r/2}+ty \big)$,   we deduce that 
 \bea 
 \mathcal{H}^2  (G(E))=\int_E(- \nu_3)^{\frac 1\alpha}= \int_A (1+|Du|^2)^{\frac 12 -\frac 1{2\alpha} }dx =\infty. \eea
  This is a contradiction as $\mathcal{H}^2  (G(E))$ can not exceed $|\mathbb{S}^2|/2$.  
  \medskip

	(b) 
	 Let  $\tilde u$ be a translation of $u$ given by  $\tilde u (x) = u(x+x')$.  Then it holds  that  $\tilde u^*(y) = u^*(y) -\langle y,x'\rangle$, and hence 
	we may assume that $x'=0$ and $0\in \mathrm{int}( \Omega)$. Thus  the section $S_{t,y'}^{u^*}$ becomes a sub-level set and it suffices to show that every sub-level set is compactly included in $\mathrm{int}(  \Omega^*)$. 
		\smallskip
		
	Suppose to the contrary that  the second assertion is false. Then   there is a sequence  $\{y_j \}$  in $ \mathrm{int}( \Omega^*)$ such that 
	\be\left\{
	\ba
	&\hbox{either  $y_j \to y\in  \p \Omega^*$\,\, or \,\,$|y_j|\to \infty$}, \\ 
	&\hbox{and $ u^*(y_j) \le C$ for all  $ j\ge 1,$ with some $C>0$. }
	\ea\right.
	\ee
	   By the first assertion (a) of this lemma,  the first case can not happen. i.e.  there is     no sequence  $\{y_j \}$     in $ \mathrm{int}( \Omega^*)$  such that   $y_j \to y\in  \p \Omega^*$ and $ u^*(y_j) \le C$  for all $j\geq1$.  Now, let us  assume that   there is     a sequence  $\{y_j \}$     in $ \mathrm{int}( \Omega^*)$  such that $|y_j|\to \infty$ and  $ u^*(y_j) \le C$  for all  $j\geq1$. Here, we may assume that $y_j/|y_j| \to \omega $ for some $ \omega \in \mathbb{S}^1$  by extracting a subsequence. 
  Then  it holds that   
  \be
  \langle y_{j}/|y_{j}|, \omega \rangle \ge 1/2\quad\hbox{ for     large $j$.}
  \ee
  Since    $B_\gamma(0)   \subset \Omega$ for some $\gamma>0$,  we have  $\gamma \omega  \in \Omega$ and hence it follows that  
	\be
	 u^*(y_j)\ge \langle y_j, \gamma  \omega  \rangle - u(\gamma \omega ) \ge \frac{\gamma |y_j|}{2}-u(\gamma \omega ) \to \infty \quad\hbox{  as $j\to \infty$. }
	 \ee
	 This is a contradiction to the hypothesis that $ u^*(y_j) \le C$ for all $j\geq1$. Thus  we conclude that every sub-level set is compactly included in $\mathrm{int}(  \Omega^*)$.  
This	 finishes the proof. 
\end{proof}

\bigskip

\section{Regularity and global structure}\label{sec-growth}

This section is  devoted to the proof of  Theorem  \ref{thm-smooth}, concerning strict convexity, smoothness and growth rate of given translator.   
 
   The main task  is to establish  the growth rate  of a  height function $u$ as in \eqref{eq-440}.   Instead of $u$,  we work with its Legendre transform $u^*$   and show  $u^*$ has the growth rate $|x|^{\frac{1}{2\alpha}}$.  The advantages of dual equation 
     \bea 
      \det D^2 u^* = (1+|x|^2)^{\frac{1}{2\alpha}-2} 
      \eea are the fact that the right-hand side  depends only on the position on the domain and it satisfies a doubling  condition; see Definition \ref{def-doubling}.     
 
 The regularity theory  of  Caffarelli   for   the  Monge--Amp\`ere measure  with  the doubling condition  can be applied to \eqref{eq-dualtranslator-1}. In particular,  the theorem of  Caffarelli (Theorem \ref{thm-caffarelli}) reveals that  the level sets of   solutions are balanced around base points and this  is a  key ingredient that reduces the question on the growth rate to a question on the eccentricities of level sets.

 To show  the eccentricities of level sets are uniformly bounded for large heights, we employ the result   \cite{DSavin} by Daskalopoulos and Savin on the homogeneous equation 
 \be \label{eq-homo-DS}
 \det D^2 v= |x|^{\frac1\alpha -4} 
 .\ee
Note this serves as an asymptotic equation to  \eqref{eq-dualtranslator-1} at high levels.
According to  \cite{DSavin},  if the level set of $v$ has a sufficiently large  eccentricity at certain height, then $v$ must exhibits a non-homogeneous growth rate near the origin. The upshot  is that a large eccentricity at a high level yields  much  larger eccentricities at lower levels.   Considering  an entire solution to  \eqref{eq-dualtranslator-1} (or \eqref{eq-homo-DS}), this  dictates a uniform radial behavior of $u^*$ at  high levels, giving the desired growth rate as in Theorem \ref{thm-equalgrowth}.  

Regarding similar estimates on the growth rate,  we mention \cite[Theorem 1.1 (ii)]{JW}, where they previously showed $u^*$ has a polynomial growth $ |x|^{1+a} \lesssim u^*(x) \lesssim |x|^{1+b}$ for large $|x|$ with some positive constants $a$ and $b$.

  In order to apply the  theory  of  Monge--Amp\`ere equations to $u^*$, 
 one needs to make sure that    sub-level sets of a solution $u^*$  are compactly included in the domain  $\mathrm{int}(\Omega^*)$ of $u^*$. Here, we  recall that  it is not yet known that $\Omega^*=\mathbb{R}^2$.  In the next lemma, we    present   some  preliminary  results  on  the Legendre dual $u^*$.

  \bigskip
  
\begin{lemma}\label{lem-dualsetup}  

Let $\Sigma $ be a translator  under  the $\alpha$-GCF  in the     geometric Alexandrov sense represented  by 
\be\label{eq-sigm-repre}
 \Sigma = \p \left\{ (x_1,x_2,x_3)\in \Omega\times  \mathbb{R}:  x_3 >u(x_1,x_2) \right\}. 
\ee 
Let $u^*$ be the Legendre dual  of $u$ given by 
$u^*(y):= \sup_{x\in \Omega } \left\{\la y,x \ra - u(x)\right\} $
 defined on $\Omega^* := \{ y\in \mathbb{R}^2: u^*(y)<\infty\}$. Then the following statements hold.  
  
 \begin{enumerate}[(a)]
\item  $u^*$ is a convex function on a  convex  set   $\Omega^*$. 

\item  $u^*$  solves  the dual equation $\det D^2 u^* = (1+|y|^2)^{\frac{1}{2\alpha}-2}$ on $\mathrm{int}(\Omega^*)$ in the  Alexandrov sense. 

\item If $0\in \mathrm{int} (\Omega)$, then  every sub-level set $\{ y\in \Omega^*: u^*(y) \le t\}$  is  compactly included in $\mathrm{int}(\Omega^*)$ 
\end{enumerate}
\begin{proof}
(a) is a direct consequence of  the convexity of $y \mapsto \sup_{x\in \Omega} \la y,x\ra$.  (b) and (c)  follow from  Proposition \ref{prop-C51} and  Theorem \ref{thm-section}, respectively. 	
\end{proof}
\end{lemma}

\smallskip

 Let us  introduce some   definitions and notations.

\begin{definition}[Doubling measure] \label{def-doubling}
A measure $\mu $ on $\Omega \subset  \mathbb{R}^n$ is said to  satisfy the doubling condition if there is a constant $C>1$ such that $ \mu (E) \le C\mu\left(\tfrac 12 E\right)  $ for any ellipsoid $E\subset \Omega$. Here, for a given ellipsoid $E$ centered at $x$,   we denote \be
 \tfrac 12 E:=\left\{ \tfrac 12(y-x)+ x \in\Omega \,:\, y\in E  \right\}, 
 \ee
 and    $C$  is called a \textit{doubling constant}.  
\end{definition}

\begin{remark}\label{remark-doubling}
It should be noted that  for $\alpha<1/3$,  a measure $(1+|x|^2)^{\frac{1}{2\alpha}-2} dx$  on $\mathbb{R}^2$  satisfies the doubling condition  and the doubling constant depends on $\alpha$; we refer to \cite[Section 3]{JW} for the proof. 
\end{remark}
\medskip

  The following theorem is a well-known result by Caffarelli about   geometric properties of sections.   We  give a proof   using a lemma in \cite{Caffarelli-boundary} as we  could not find it in the literature.
\begin{theorem}[Caffarelli]    \label{thm-caffarelli}
Let  $f(x) \,dx$ be  a doubling measure on an open set $\mathcal{O} \subset \mathbb{R}^2 $, and let    $v$ be  an Alexandrov solution to  $\det D^2v= f$   on  $\mathcal{O}  \subset \mathbb{R}^2$.  Suppose that  for $x_0\in \mathcal{O} $,  $ p\in \p v(x_0)$   and $t>0$, the   section $  S_{t,x_0} $
  is compactly included in $\mathcal{O}$. Then, there is a  symmetric  matrix $A_t$ with $\det A_t=1$ such that 
  \bea\label{eq-centerellipsoid} 
  k_0^{-1} A_t B_r \subset S_{t,x_0}-x_0  \subset k_0 A_t B_r \qquad \text{for} \quad r= t \left[\int_{S_{t,x_0}}f(x)\, dx\right]^{-1/2}.
  \eea 
Here, a constant $k_0>0$ depends only on  the doubling constant of the measure $f(x) \, dx$.

\begin{proof} By subtracting an affine function from $v$, we may assume  that 
\be
\hbox{$p=0\in \mathbb{R}^2$,\quad $  v(x_0)=\inf_\mathcal{O}  v= -t$, \quad and \quad $S_{t,x_0}=\left\{x\in\mathcal{O}\,: v (x)\le 0\right\}$.}
\ee 
Moreover,   we may use   John's lemma  and  apply an affine transform $x\mapsto Ax +b$ for some  symmetric matrix $A$ with $\det A=1$ and $b\in \mathbb{R}^2$ to assume that
 \bea\label{eq-johns} 
 B_R \subset S_{t,x_0}  \subset B_{2R}\quad \text{ for some } R>0.
 \eea  
 Here, we note that $ S_{t,x_0}  $ with center of mass $0$   is compactly included in $\mathcal{O}$  .

 		By \cite[Lemma 3]{Caffarelli-boundary}, there are  constants $c_1>1$ and $\lambda\in(0,1)$ depending   only on the  doubling constant such that 	
	\bea \label{eq-431}
	 c_1^{-1} t^2  \le R^2  \int_{S_{t,x_0}} f(x) dx \le c_1 t^2  \quad\hbox{and} \quad B_{\frac{t}{c_1R}} \subset \p v (\lambda S_{t,x_0}) \subset B_{\frac{c_1 t}{R}}.
	 \eea  
Since   $  v(x_0)=\inf_\mathcal{O}  v$,  the second assertion above implies that   
\be \label{eq-x_0-ldS}
x_0 \in \lambda S_{t,x_0}.
\ee

\smallskip

 Next, we claim that 
  for every $y\in \partial S_{t,x_0}$,  
  \be\label{eq-dist-S-lS}
  \dist (y,\lambda \partial S_{t,x_0})\ge (1-\lambda)R.
  \ee
   Indeed, since $S_{t,x_0}$ is convex and it contains $B_R$, there exists $\nu' \in \mathbb{S}^1 $ such that 
   \be
 \sup_{z\in \partial S_{t,x_0}} \langle z , \nu' \rangle = \langle y ,\nu '\rangle   \quad\hbox{    and \quad $\langle y, \nu'\rangle \ge R$.}
  \ee 
  This implies that 
\begin{align}
\inf_{z'\in \lambda \partial S_{t,x_0}}| y-z'|&\ge \inf_{z' \in \lambda \partial S_{t,x_0}} \langle y-z' , \nu' \rangle \\
&= \langle y,\nu'\rangle - \lambda \sup_{z\in \partial S_{t,x_0}}\langle z,\nu' \rangle=(1-\lambda)\langle y,\nu '\rangle \ge (1-\lambda)R.
\end{align}
This yields \eqref{eq-dist-S-lS}.

 Using   \eqref{eq-x_0-ldS} and  \eqref{eq-dist-S-lS}, we find   
a constant $c_2>1$ depending on $\lambda$ such that 
\be B_{\frac{R}{c_2} } \subset  S_{t,x_0}-x_0  \subset  B_{c_2R}.
\ee  
In light of   \eqref{eq-431} and the definition of $r$,  we deduce that
\be 
c_3^{-1}B_{ r} \subset S_{t,x_0} -x_0 \subset  c_3B_{ r}  ,
\ee
where $c_3=c_1^{\frac{1}{2}}c_2$. This finishes the proof. 
\end{proof}

\end{theorem}

\smallskip

\begin{remark}\label{remark-fol}
 
In Theorem \ref{thm-caffarelli}, a symmetric matrix $A_t$ can be decomposed as  $A_t=Q\Lambda Q^{T}$ with    a diagonal matrix $\Lambda $  and an orthogonal matrix $Q$. So    a symmetric matrix $A_t $  in \eqref{eq-centerellipsoid}  can be replaced by $Q\Lambda $  since $Q^{T}B_r=B_r$.
\end{remark}

Finally, we recall some   notations from  \cite{DSavin}. 

\begin{definition}\label{def-eccen-section}
(i) For a given section $S_{t,x_0}$ and a   matrix $A_t $ with $\det A_t =1$  (not necessarily symmetric), we write
\be
S_{t,x_0} \sim A_t \qquad  \text{ if }\quad \text{the condition \eqref{eq-centerellipsoid} holds. }\
\ee
In this case, we   say that  the eccentricity of the section $S_{t,x_0}$ is proportional to $|A_t|$. Here, $|A|=\sup_{|v|=1} |Av| $. See Remark \ref{rmk-def67}.  As $k_0$ in \eqref{eq-centerellipsoid} depends on the doubling constant of $f(x)\, dx$, we write $\sim_f$ in place of $\sim$ whenever an indication is needed.

\smallskip

(ii) For two regions $\Omega_1$  and $\Omega_2$ in $\mathbb{R}^2$ and a constant $0<\delta<1$, we write 
\be
\Omega_2 \in \Omega_1 \pm \delta\qquad  \text{ if }\quad (1-\delta)\Omega_1\subset \Omega_2 \subset (1+\delta ) \Omega_1 .
\ee
 
\end{definition}
\medskip

\begin{remark}\label{rmk-def67} From the definition (i), it is possible that $S_{t,x_0} \sim A$ and $S_{t,x_0}\sim A'$ for two different unimodular matrices. In this case, however, $C^{-1}|A'|\le|A|\le C |A'|$ for some uniform $C$ which depends only on the doubling constant. For this reason, we denoted that the eccentricity is `proportional' to $|A_t|$. Indeed, the eccentricity  of   an open bounded   convex set $S\subset \mathbb{R}^2$ can be uniquely defined as $ \det(T)^{-1/2} |T|,$ for normalizing affine transform $T$; see \cite[Definition 4.1]{MA-figalli-book} for the normalization.

\end{remark}

\smallskip

 The following key result roughly shows that a sufficiently high eccentricity is unstable in the sense that it increases as the level decreases. The proof is an application of Lemma 3.1 and 3.2 in \cite{DSavin}. 

\begin{theorem}\label{thm-controleccen} For $\alpha \in (0,1/4)$ and a given measure $f(x)\, dx$ with doubling constant $\mu_0<\infty$, there exist constants $\e_0>0$, $C_0<\infty $, and $\tau_0\in (0,1)$ depending on $\alpha$ and $\mu_0$ with the following significance. Let $v$ be an Alexandrov solution to $\det D^2 v =  f(x) $ on an open set $\mathcal{O}$    containing the origin. Suppose $v(0)=0$, $0 \in \partial v(0)$, $S_{1} \subset\subset \mathcal{O}$  and 
\begin{equation}
    |f(x)-|x|^{\frac{1}{\alpha }-4}|\le \e _0 (1+|x|^{\frac 1 {\alpha} -4 }) \quad \text{ on } S_1.
\end{equation}
If $S_1\sim_f A_1$ for some $|A_1|\ge C_0$, then $|A_{\tau_0}|\ge 2 |A_1|$ for all $A_{\tau_0}$ with $S_{ \tau_0 }\sim_f A_{\tau_0}$. 
 
\begin{proof}Let us assume $S_1 \sim A_1$ so that \eqref{eq-centerellipsoid} is satisfied for $r$ as in the equation with $t=1$. From Remark \ref{remark-fol}, we may choose 
\begin{equation}
A_{1} = P
   \begin{pmatrix} b& 0 \\ 0 &  1/b
   \end{pmatrix},
\end{equation}
for some orthogonal matrix $P$ and $b=|A_1|\ge1$. Then $
     \hat  v(x):= v\left(A_1  r  x\right)$ is an Alexandrov solution to $ \det D^2 \hat v (x)= r ^4 f(A_1rx)$. First, observe that $rb \ge r_0 $ for some positive constant $r_0=r_0(\alpha,\mu_0)$. Indeed, a direct computation which uses \eqref{eq-centerellipsoid} yields  $ cr\le r^{-1}(rb)^{2-\frac{1}{2\alpha}} \le Cr $ for some positive constants $c=c(\alpha,\mu_0)$ and $C=C(\alpha,\mu_0)$. This implies $(rb)^{-\frac1{2\alpha}}\le r^{-2} (rb)^{2-\frac{1}{2\alpha}}\le C$ and the bound follows. Next, note that $|A_1|=b\ge C_0$ and  
\begin{align}
&      (rb)^{4-\frac1\alpha}f (A_1rx) -  |x_1|^{\frac{1}\alpha -4 }\\
 &=(rb)^{4-\frac1\alpha} (f(A_1rx)-|A_1rx|^{\frac{1}\alpha -4})+(|x_1|^2+|x_2|^2 {b^{ -4}})^{\frac{1}{2\alpha }-2} - |x_1|^{\frac{1}{\alpha}-4} \\
&\le  (rb)^{4-\frac1\alpha}  \e_0 (1+|A_1rx|^{\frac1\alpha-4} ) + \e_1(b)\le  {  \e_0(r_0)^{4-\frac1{\alpha}}(1+ k_0^{\frac{1}{\alpha}-4}) + \e_1(b),}
\end{align}     
where $\e_1(b)\to 0$ as $b \to \infty$. This shows that for every $\e>0$ there is small $\e_0$ and large $C_0$ so that $\det D^2\hat v = r^{\frac1\alpha}b^{\frac1\alpha -4} \tilde f(x)$ with $|\tilde f(x) -|x_1|^{\frac{1}{\alpha}-4}|\le \e $ on $S^{\hat v}_1$. 
By \cite[Lemma 3.2]{DSavin}, for a given $\theta_0>0$, there exist $\e_0>0$, $C_0<\infty$   and $t_1=t_1\in (0,1)$ depending on $\theta_0$, $\alpha$, $\mu_0$,  such that 
\bea S^{\hat v}_{t_1} \in B_0 D_{t_1} (\Gamma \pm \theta_0), \text{ where } \Gamma:= \{(x_1, x_2)\subset \mathbb{R}^2\,:\,  x_1^{\frac1\alpha -2} +x_2^2 <1 \}. \eea
Here, $ D_t := \begin{pmatrix} t^{\frac{\alpha }{1-2\alpha}} &  0  \\ 0 &  t^{\frac12}   \end{pmatrix}$ and $B_0$ is a lower triangular matrix  $B_0:= \begin{pmatrix} b_{0,11}&  0  \\ b_{0,21} &   b_{0,22}
   \end{pmatrix}$  with bounded entries $C^{-1}\le |b_{0,ii}| \le C$ and $|b_{0,21}|\le C$ for some $C=C(\alpha,\mu_0)$.  
\medskip

As a next step, consider $ v_1 (x):= t_1^{-1}\hat v( B_0D_{t_1} x)$ and observe that $S^{v_1}_1\in \Gamma \pm \theta_0$.  
We aim to apply \cite[Lemma 3.1]{DSavin} several but definite amount, say  $k=k(\alpha,\mu_0)\in \mathbb{N}$, times to $v_1$. Let $k$ be an undetermined integer that will be fixed at the end. As a consequence, there exist $\theta_0=\theta_0(\alpha,\mu_0)>0$,  $\e_0=\e_0(\alpha,\mu_0, k)>0$ and $C_0=C_0(\alpha,\mu_0,k)<\infty $  such that 
\bea  S^{\hat v}_{t_1 t_0^k}\in B_k D_{t_1 t_0^k}(\Gamma \pm \theta_0 t_0^{k\delta}). \eea
Here, $B_k$ are  lower triangular matrices  $B_k:= \begin{pmatrix} b_{k,11}&  0  \\ b_{k,21} &   b_{k,22}
   \end{pmatrix}$  with bounded entries $C^{-1}\le |b_{k,ii}| \le C$ and $|b_{k,21}|\le C$ for some $C=C(\alpha,\mu_0)$, and $t_0 \in (0,1)$, $\delta>0$ are universal constants appearing in  \cite[Lemma 3.1]{DSavin}. The smallness of $\theta_0=\theta_0(\alpha,\mu_0)$ is needed to guarantee the universal bound on the entries of $B_k$. See the proof of \cite[Proposition 3.3]{DSavin} for more details on this argument.

 In terms of the original function $ v$, this amounts to say that  
 \be
 (r A_{1}B_kD_{t_1t_0^k} )^{-1} S_{t_1t_0^k}  \in   \Gamma \pm \theta_0  t_0^{k\delta}.\ee  Namely, the eccentricity of $S_{t_0t_1^k}$ is proportional to $  {\det (A_{1} B_k D_{t_0t_1^k})^{-\frac12}}    {|A_{1} B_k D_{t_0t_1^k}|}$. 
   Since   $\det (A_{1} B_k D_{t})$ is proportional to $\det D_{t}={t}^{    \left(\frac{\alpha}{1-2\alpha} +\frac12 \right)} $,   it holds that for $t=t_1 t_0^k$
 \bea 
 {\det (A_{1} B_k D_t)^{-\frac12}}    {|A_{1} B_k  D_t|} &=   {\det (A_{1} B_k D_t)^{ -\frac12}}   {\left |\begin{pmatrix} b &0 \\ 0& b^{-1} \end{pmatrix} B_k D_t\right|}  \\
&\ge c_1\,  {t}^{     -\frac12 \left(\frac{\alpha}{1-2\alpha} +\frac12 \right)}\left |\begin{pmatrix} b &0 \\ 0& b^{-1} \end{pmatrix} B_k \begin{pmatrix} t^{\frac{\alpha }{1-2\alpha}}&0 \\ 0& t^{\frac12  }\end{pmatrix}\right| \\
&\ge c_1\, b\, {t}^{     -\frac12 \left(\frac{\alpha}{1-2\alpha} +\frac12 \right)}t^{ \frac{\alpha }{1-2\alpha}} =  c_1\, b\, {t}^{     \frac12 \left(\frac{\alpha}{1-2\alpha} -\frac12 \right)}, 
\eea  
for some constant $c_1=c_1(\alpha,\mu_0)>0$ . As $\frac{\alpha}{1-2\alpha} -\frac12 <0$, we may fix $k=k(\alpha,\mu_0)$  so that if $S_{t_1t_0^k} \sim A'$ then $|A'|\ge 2b=2|A_1|$. We choose $\tau_0=t_1 t_0^k$ and this finishes the proof.
\end{proof}
\end{theorem}

\bigskip

Now, we prove the main theorems in this paper.   
   
\begin{proof}[Proof of Theorem \ref{thm-equalgrowth}]
  We first prove under the assumption  $0\in \mathcal{O}$, $v(0)=0$ and $0\in \partial v(0)$. In view of Theorem \ref{thm-caffarelli}, it suffices to show that the eccentricity is uniformly bounded for large $l\ge l_1$. i.e. we show if $S_l \sim A_l $ for $l\ge l_1$, then $|A_l| <C_0$, where $C_0$ is the universal constant appearing in Theorem \ref{thm-controleccen} and $l_1$ is a constant depending on $v$. 
  
  Observe $v_\lambda(x)= \lambda^{-\frac{1}{2\alpha} }v(\lambda x)$ solves $\det D^2 v_\lambda(x) = f_\lambda(x):= \lambda^{-\frac{1}{\alpha}+4}f(\lambda x)$. For some $\eps>0$ to be chosen later, if $f(x)$ satisfies \eqref{eq-RHS.conv2}, then there exists $\lambda_0=\lambda_0(f,\varepsilon)$ such that 
  \[|f_{\lambda}(x)-|x|^{\frac{1}{\alpha}-4} | \le 2\varepsilon (1+|x|^{\frac{1}{\alpha}-4}), \quad \text{ on }x\in\mathbb{R}^2,\]
for all $\lambda \ge \lambda_0$. Using $\varepsilon_0$ obtained in Theorem \ref{thm-controleccen},  let us choose $\varepsilon=\varepsilon_0/2$. By Theorem \ref{thm-controleccen}  there is universal $\tau_0\in(0,1)$ such that if $S_{l}\sim A_{l}$ for some $l\ge \lambda_0^{1/{2\alpha}}$  with $|A_l| \ge C_0$, then $|A_{\tau_0 l }| \ge 2|A_l|$. Assume on the contrary to our initial goal, suppose there is a sequence $l_i \to \infty$ such that the eccentricity of $S_{l_i}
  \sim A_{l_i}$ with $|A_{l_i}| \ge C_0$. Then $|A_{\tau_0^k l_i}| \ge 2^k C_0$ whenever $S_{\tau_0^k l_i} \sim A_{\tau_0^k l_i}$ and ${\tau_0^{k-1} l_i}\ge l_0$. Since $l_i \to \infty$, this implies for every $M<\infty$ there is $l\in (\tau_0l_0, l_0)$ and unimodular $A_{l}$ such that $S_l \sim A_l$ and $|A_l|\ge M$. This is impossible as otherwise $S_{l_0}$ contains a sequence of ellipsoids with a uniform lower bound on inside area $4^{-1} \mathrm{Area}(S_{\tau_0l_0})$ (here $4^{-1}$ comes from John's ellipsoid)  while having unbounded eccentricities. This proves the statement.

In general case, we fix a point $x_0 \in \mathcal{O}$ such that $0 \in \partial v(x_0)$, which is possible by the compactness assumption on sub-level sets. Let us   define $\hat v (x) := v(x+x_0)-v(x_0)$ to satisfy $\hat v(0)=0$ and $0\in \p \hat v(0)$.  It suffices to show \eqref{eq-CR} for $\hat v$.
 Note $\hat v$ is solves  $ \det D^2 \hat  v(x) =f(x+x_0)$ on $\mathcal{O}-x_0$,
whose right hand side has the same doubling constant. Observe 
 if $f(x+x_0)$ is viewed as a function of $x$, then it also satisfies the condition \eqref{eq-RHS.conv2}, and  $\hat v_\lambda (x)= \lambda ^{-\frac{1}{2\alpha}} \hat v(\lambda x)$ satisfies $\det D^2 \hat v_\lambda =f_\lambda(x+\lambda^{-1} x_0)$.  Therefore, the same proof works except  that, at this time, $l_0$ is a constant of $\alpha$, $f$, and $x_0$.  
\end{proof}

\bigskip

\begin{proof}[Proof of Theorem \ref{thm-smooth}]
Suppose that     $\Sigma $  is represented  by  \eqref{eq-sigm-repre} 
for   a convex function $u$  over a convex   domain  $\Omega\subset \mathbb{R}^2$. Here, we may assume that $0\in \mathrm{int} (\Omega)$.    
By Lemma \ref{lem-dualsetup} and Theorem \ref{thm-equalgrowth}, the Legendre dual  $u^*$ of $u$ is an   entire solution to   the dual equation \eqref{eq-dualtranslator-1} on $\mathbb{R}^2$ in the Alexandrov sense.  Furthermore,     it holds   from  \eqref{eq-CR}  that
\be\label{eq-smooth-dual-g-r}
C^{-1} |y|^{\frac{1}{2\alpha}} \le u^*(y) \le C |y|^{\frac{1}{2\alpha}},
\ee for all  $|y|\ge R$,
 with some   $R>0$ and a universal constant  $C$. 
\smallskip

 For large $l>0$,  restrictions of $u^*$   can be viewed as  an Alexandrov solution to   \eqref{eq-dualtranslator-1}  on  $ Z_l:=\{y\in\mathbb{R}^2: u^*(y) <  l\}$ with the  Dirichlet boundary condition:   $u^* =l$ on  $\p Z_l$. 
  Thus, we deduce that   $u^*$ is    smooth    strictly convex   on $\mathbb{R}^2$  by employing  the interior regularity of the Monge--Amp\`ere equation;  see \cite{caffarelli1990interior, MA-Gutierrez, MA-figalli-book} for instance. Therefore, we conclude that  $ u = (u^*)^*$  is also an entire  smooth  strictly convex    solution  to  \eqref{eq-translatorgraph} with the specific growth  in   \eqref{eq-440} in light of \eqref{eq-smooth-dual-g-r}.
\end{proof}

\bigskip

\begin{proof}[Proof of Corollary \ref{cor-quantitativestability}]
By an inductive application of Theorem \ref{thm-controleccen}, there exists $l_0$, $\tau_0$ and $C_0$ such that if $S_{l'}\sim A_{l'}$ with $|A_{l'}|\le  M$ for some $l'\ge l_0$ and $M\ge C_0$, then there hold  $|A_{\tau_0^{-k} l'}|\le M$ for all $S_{\tau_0^{-k} l'} \sim A_{\tau_0^{-k} l'}$ and  positive integers $k$. This proves the statement for discrete levels $\tau_0^{-k} l'$ for constant $C_1=1$. For every other $l\ge l'$, there is a unique $k'\ge 0 $ such that $ l\in [\tau_0^{-k'} l', \tau_0^{-(k'+1)} l')$. In view of trivial inclusion 
$ S_{\tau_0^{-k'} l'}\subset S_{l}\subset S_{\tau_0^{-(k'+1)} l'}$, we infer that the statement holds for a larger universal $C_1\ge 1$. 
\end{proof}

\bigskip

\subsection*{Acknowledgments}
BC thanks KIAS for the visiting support and also thanks University of Toronto where the research was initiated. BC and KC were supported by  the National Research Foundation(NRF) grant funded by the Korea government(MSIT) (RS-2023-00219980). BC has been supported by NRF of Korea grant No. 2022R1C1C1013511, POSTECH Basic Science Research Institute grant No. 2021R1A6A1A10042944, and a POSCO Science Fellowship. KC has been supported by the KIAS Individual Grant MG078902, a POSCO Science Fellowship, and an Asian Young Scientist Fellowship. SK has been partially supported by NRF of Korea grant No. 2021R1C1C1008184. We are grateful to Liming Sun for fruitful discussion. We would like to thank Simon Brendle, Panagiota Daskalopoulos, Pei-Ken Hung, Ki-ahm Lee, Christos Mantoulidis, Connor Mooney for their interest and insightful comments.

 \bibliography{GCF-ref.bib}
\bibliographystyle{abbrv}

\end{document}